\documentclass[10pt]{amsart}
\usepackage{indentfirst,enumerate,cite,amssymb,amsfonts,amsmath,amsthm,mathrsfs,dsfont}
\usepackage{color}
\usepackage{multicol}
\usepackage[colorlinks=true,linkcolor=blue,citecolor=blue]{hyperref}
\usepackage{enumerate}
\usepackage{geometry}
\numberwithin{equation}{section}

\theoremstyle{plain}
\newtheorem{thm}{Theorem}[section]
\newtheorem{prop}[thm]{Proposition}

\newtheorem{defi}[thm]{Definition}
\newtheorem{rem}[thm]{Remark}

\newtheorem{hyp}[thm]{Hypothesis}

\newcommand{\Rep}{\mbox{Rep}}
\newcommand{\HS}{{\mathtt{HS}}}
\newcommand{\Tr}{\mbox{\emph{Tr}}}

\newcommand{\N}{{\mathbb{N}}}

\newcommand{\R}{\mathbb{R}}
\newcommand{\Z}{\mathbb{Z}}
\newcommand{\C}{\mathbb{C}}

\newcommand{\St}{{\mathbb S}^3}

\newcommand{\DG}{\mathcal{D}'(G)}

\newcommand{\jp}[1]{{\left\langle{#1}\right\rangle}}

\newlength{\dhatheight}
\newcommand{\doublehat}[1]{%
	\settoheight{\dhatheight}{\ensuremath{\hat{#1}}}
	\addtolength{\dhatheight}{-0.15ex}
	\widehat{\vphantom{\rule{5pt}{\dhatheight}}%
		\smash{\widehat{#1}}}}

\begin{document}

%
%
%
%
%
%
%
%
%


\title[Global Properties of Vector Fields in Komatsu classes II]{Global Properties of Vector Fields on Compact \\  Lie Groups in Komatsu classes. II. Normal Forms}


\author[Alexandre Kirilov]{Alexandre Kirilov}
\address{
	Universidade Federal do Paran\'{a}, 
	Departamento de Matem\'{a}tica,
	C.P.19096, CEP 81531-990, Curitiba, Brazil
}
\email{akirilov@ufpr.br}


\author[Wagner de Moraes]{Wagner A. A. de Moraes}
\address{
	Universidade Federal do Paran\'{a},
	Programa de P\'os-Gradua\c c\~ao de Matem\'{a}tica,
	C.P.19096, CEP 81531-990, Curitiba, Brazil
}
\email{wagneramat@gmail.com}


\author[Michael Ruzhansky]{Michael Ruzhansky}
\address{Ghent University, 
	Department of Mathematics: Analysis, Logic and Discrete Mathematics, 
	Ghent, Belgium 
	and 
	Queen Mary University of London, 
	School of Mathematical Sciences, 
	London, United Kingdom
}
\email{Michael.Ruzhansky@ugent.be}




\begin{abstract}
Let $G_1$ and $G_2$ be compact Lie groups, $X_1 \in \mathfrak{g}_1$, $X_2 \in \mathfrak{g}_2$ and consider the operator
\begin{equation*}
L_{aq} = X_1 + a(x_1)X_2 + q(x_1,x_2),
\end{equation*}
where  $a$ and $q$ are ultradifferentiable functions in the sense of Komatsu, and $a$ is real-valued.  We characterize completely the global hypoellipticity and the global solvability of $L_{aq}$ in the sense of Komatsu. For this, we present a conjugation between $L_{aq}$ and a constant-coefficient operator that preserves these global properties in Komatsu classes. We also present examples of globally hypoelliptic and globally solvable operators on $\mathbb{T}^1\times \St$ and $\St\times \St$ in the sense of Komatsu. In particular, we give examples of differential operators which are not globally $C^\infty$--solvable, but are globally solvable in Gevrey spaces.
\end{abstract}

\maketitle
\tableofcontents

\section{Introduction}

The present paper is a continuation of our paper \cite{KMR19c} where we have characterized the global hypoellipticity and global solvability in the sense of Komatsu (of Roumieau and Beurling types) of constant-coefficients vector fields defined on compact Lie groups, and the influence of lower-order perturbations in the preservation of these properties. 

In this paper, we present a class of first-order operators with variable coefficients that can be reduced to a constant-coefficient operator employing a conjugation. Such a reduction ensures that the original operator and the conjugated constant-coefficient operator have the same type of global properties in Komatsu sense. This equivalence was inspired in reduction to normal forms, which is a technique widely used in this context, see for example \cite{AGK18,AGKM18,CC00,DicGraYosh02,Hou82,Pet11}. It should be emphasized here that, as far as we know, this is the first time that this technique is extended to ultradifferentiable functions and ultradistributions in Komatsu classes, especially in the Beurling setting. In the case of Gevrey spaces of Roumieu type, this technique was used in \cite{AKM19} and \cite{BDG18} to reduce vector fields (and systems of vector fields) defined on tori to their normal forms. In the case of Lie groups, this reduction was firstly used in \cite{KMR19b} in the smooth and distributional cases.

The definition of such a conjugation depends on the characterization of ultradifferentiable functions and ultradistributions through their partial Fourier series in Komatsu classes, which is done in Section 3. With this characterization in place, in Section 4, we use the properties of the Komatsu classes to show that the conjugation is well defined in Komatsu classes, in Roumieu and Beurling settings. Finally, we obtain the normal form of the given operator and, in Section 5, we provide the characterization of global hypoellipticity and global solvability in the sense of Komatsu.

We conclude the paper presenting new examples of globally hypoelliptic and globally solvable operators, in the sense of Komatsu, on $\mathbb{T}^1\times \St$ and $\St\times \St$. In particular, we construct an example of operator with variable coefficients that is neither globally hypoelliptic in the sense of Komatsu, nor globally solvable in $C^\infty$--sense, but it is globally solvable in Gevrey spaces.

\section{Preliminaries}

In this section, we recall most of the notations and preliminary results necessary for the development of this study. A very careful presentation of these concepts and the demonstration of all the results presented here can be found in the references \cite{FR16}  and \cite{livropseudo}.

Let $G$ be a compact Lie group and let $\Rep(G)$  be the set of continuous irreducible unitary representations of $G$. Since $G$ is compact, every continuous irreducible unitary representation $\phi$ is finite dimensional and it can be viewed as a matrix-valued function $\phi: G \to \C^{d_\phi\times d_\phi}$, where $d_\phi = \dim \phi$. We say that $\phi \sim \psi$ if there exists an unitary matrix $A\in C^{d_\phi \times d_\phi}$ such that $A\phi(x) =\psi(x)A$, for all $x\in G$. We will denote by $\widehat{G}$ the quotient of $\Rep(G)$ by this equivalence relation.

For $f \in L^1(G)$ the group Fourier transform of $f$ at $\phi \in \Rep(G)$ is
\begin{equation*}
\widehat{f}(\phi)=\int_G f(x) \phi(x)^* \, dx,
\end{equation*}
where $dx$ is the normalized Haar measure on $G$.
By the Peter-Weyl theorem, we have that 
\begin{equation}\label{ortho}
\mathcal{B} := \left\{\sqrt{d_\phi} \, \phi_{ij} \,; \ \phi=(\phi_{ij})_{i,j=1}^{d_\phi}, [\phi] \in \widehat{G} \right\},
\end{equation}
is an orthonormal basis for $L^2(G)$, where we pick only one matrix unitary representation in each class of equivalence, and we may write
\begin{equation*}
f(x)=\sum_{[\phi]\in \widehat{G}}d_\phi \emph{\Tr}(\phi(x)\widehat{f}(\phi)).
\end{equation*}
Moreover, the Plancherel formula holds:
\begin{equation*}
\label{plancherel} \|f\|_{L^{2}(G)}=\left(\sum_{[\phi] \in \widehat{G}}  d_\phi \ 
\|\widehat{f}(\phi)\|_{\HS}^{2}\right)^{1/2}=:
\|\widehat{f}\|_{\ell^{2}(\widehat{G})},
\end{equation*}
where 
\begin{equation*} \|\widehat{f}(\phi)\|_{\HS}^{2}=\emph{\Tr}(\widehat{f}(\phi)\widehat{f}(\phi)^{*})=\sum_{i,j=1}^{d_\phi}  \bigr|\widehat{f}(\phi)_{ij}\bigr|^2.
\end{equation*}

The group Fourier transform of $u\in \DG$ at a matrix unitary representation $\phi$ is the matrix $\widehat{u}(\phi) \in \C^{d_\phi \times d_\phi}$, whose components are given by
$$
\widehat{u}(\phi)_{ij} = \jp{u,\overline{\phi_{ji}}},
$$
where $\jp{\,\cdot\,,\,\cdot\,}$ denotes the distributional duality.

Let $\mathcal{L}_G$ be the Laplace-Beltrami operator of $G$. For each $[\phi] \in \widehat{G}$, its matrix elements are eigenfunctions of $\mathcal{L}_G$ correspondent to the same eigenvalue that we will denote by $-\nu_{[\phi]}$, where $\nu_{[\phi]} \geq 0$. Thus
\begin{equation}\label{laplacian}
-\mathcal{L}_G \phi_{ij}(x) = \nu_{[\phi]}\phi_{ij}(x), \quad \textrm{for all } 1 \leq i,j \leq d_\phi,
\end{equation}
and we will denote by
$$
\jp \phi := \left(1+\nu_{[\phi]}\right)^{1/2}
$$
the eigenvalues of $(I-\mathcal{L}_G)^{1/2}.$ We have the following estimate for the dimension of $\phi$ (Proposition 10.3.19 of \cite{livropseudo}): there exists $C>0$ such that for all $[\xi] \in \widehat{G}$ it holds
\begin{equation*}\label{dimension}
d_\phi \leq C \jp{\phi}^{\frac{\dim G}{2}}.
\end{equation*}

For $x\in G$, $X\in \mathfrak{g}$ and $f\in C^\infty(G)$, define 
$$
L_Xf(x):=\frac{d}{dt} f(x\exp(tX))\bigg|_{t=0}.
$$

The operator $L_X$ is left-invariant, that is, $\pi_L(y)L_X = L_X\pi_L(y)$, for all $y \in G$. When there is no possibility of ambiguous meaning, we will write only $Xf$ instead of $L_Xf$. 

%

Let  $P: C^{\infty}(G) \to C^{\infty}(G)$ be a continuous linear operator. The  symbol of the operator $P$ in $x\in G$ and $\phi \in \mbox{{Rep}}(G)$, $\phi=(\phi_{ij})_{i,j=1}^{d_\phi}$ is
$$
\sigma_P(x,\phi) := \phi(x)^*(P\phi)(x) \in \C^{d_\phi \times d_\phi},
$$
where $(P\phi)(x)_{ij}:= (P\phi_{ij})(x)$, for all $1\leq i,j \leq d_\phi$, and we have
$$
Pf(x) = \sum_{[\phi] \in \widehat{G}} \dim (\phi) \mbox{Tr} \left(\phi(x)\sigma_P(x,\phi)\widehat{f}(\phi)\right),
$$
for all $f \in C^\infty(G)$ and $x\in G$.

When $P: C^\infty(G) \to C^\infty(G)$ is a continuous linear left-invariant operator, that is $P\pi_L(y)=\pi_L(y)P$, for all $y\in G$, we have that $\sigma_P$ is independent of $x\in G$ and
$$
\widehat{Pf}(\phi) = \sigma_P(\phi)\widehat{f}(\phi),
$$
for all $f \in C^\infty(G)$ and $[\phi] \in \widehat{G}$ and, by duality, this remains true for all $f \in \mathcal{D}'(G)$. For instance, by relation \eqref{laplacian}, we obtain 
\begin{equation}\label{symbollaplacian}
	\widehat{\mathcal{L}_Gf}(\phi) = -\nu_{[\phi]}\widehat{f}(\phi),
\end{equation}
for all $f\in \mathcal{D}'(G)$ and $[\phi] \in \widehat{G}$.

Let $Y\in \mathfrak{g}$. It is easy to see that the operator $iY$ is symmetric on $L^2(G)$. Hence, for all $[\phi] \in \widehat{G}$ we can choose a representative $\phi$ such that $\sigma_{iY}(\phi)$ is a diagonal matrix, with entries denoted by $\lambda_m(\phi) \in \R$, $1 \leq m \leq d_\phi$. By the linearity of the symbol, we obtain
$$
\sigma_X(\phi)_{mn} = i\lambda_m(\phi) \delta_{mn}, \quad \lambda_j(\phi) \in \R.
$$
Notice that $\{\lambda_m(\phi)\}_{m=1}^{d_\phi}$ are the eigenvalues of $\sigma_{iX}(\phi)$ and they are independent of the choice of the representative, since the symbol of equivalent representations are similar matrices. Moreover, since $-(\mathcal{L}_G - X^2)$ is a positive operator and commutes with $X^2$, we have
\begin{equation}\label{symbol}
|\lambda_m(\phi)| \leq \jp{\phi},
\end{equation}
for all $[\phi] \in \widehat{G}$ and $1 \leq m \leq d_\phi$.

Let $G_1$ and $G_2$ be compact Lie groups and set $G=G_1\times G_2$. Given $f \in L^1(G)$ and  $\xi \in {\Rep}(G_1)$, the partial Fourier coefficient of $f$ with respect to the first variable  is defined by 
$$
\widehat{f}(\xi, x_2) = \int_{G_1} f(x_1,x_2)\, \xi(x_1)^* \, dx_1 \in \C^{d_\xi \times d_\xi}, \quad x_2 \in G_2,
$$
with components
$$
\widehat{f}(\xi, x_2)_{mn} = \int_{G_1} f(x_1,x_2)\, \overline{\xi(x_1)_{nm}} \, dx_1, \quad 1 \leq m,n\leq d_\xi.
$$
Analogously we define the partial Fourier coefficient of $f$ with respect to the second variable. Notice that, by definition, $\widehat{f}(\xi,\: \cdot \:)_{mn} \in C^\infty(G_2)$ and $\widehat{f}(\: \cdot \:, \eta)_{rs} \in C^\infty(G_1)$. 

Let $u \in \DG$, $\xi \in {\Rep}(G_1)$ and $1\leq m,n \leq d_\xi$. The $mn$-component of  the partial Fourier coefficient of $u$ with respect to the first variable is the linear functional defined by
$$
\begin{array}{rccl}
\widehat{u}(\xi, \: \cdot\: )_{mn}: & C^\infty(G_2) & \longrightarrow & \C \\
& \psi & \longmapsto & \jp{\widehat{u}(\xi, \: \cdot \:)_{mn},\psi} := \jp{u,\overline{\xi_{nm}}\times\psi}_G.
\end{array}
$$
In a similar way, for $\eta \in {\Rep}(G_2)$ and $1 \leq r,s\leq d_\eta$, we define the $rs$-component of the partial Fourier coefficient of $u$ with respect to the second variable.
It is easy to see that $ \widehat{u}(\xi, \: \cdot\: )_{mn} \in \mathcal{D}'(G_2)$ and $\widehat{u}( \: \cdot\:,\eta )_{rs} \in \mathcal{D}'(G_1)$. 

Notice that
\begin{equation*}
\doublehat{\,u\,}\!(\xi,\eta)_{mn_{rs}} = \doublehat{\,u\,}\!(\xi,\eta)_{rs_{mn}} =  \widehat{u}(\xi \otimes \eta)_{ij},
\end{equation*}
with $i = d_\eta(m-1)+ r$ and $j  =  d_\eta(n-1) + s$, whenever $u \in C^\infty(G)$ or $u \in \DG$. More details about partial Fourier series in the framework of smooth functions and distributions can be found in \cite{KMR19}.

In this paper, we deal with operators and their properties in Komatsu classes. So we need to introduce some notations, results and technical lemmas that will be used in the sequel. All definitions are taken from \cite{DR16}, \cite{Kom73} and \cite{Rou60}.

Let $\{ M_k\}_{k \in\N_0}$ be a sequence of positive numbers such that there exist $H>0$ and $A\geq1$ satisfying
\begin{description}
	\item[(M.0)] $M_0=1$.
	\item[(M.1)] (stability)	$M_{k+1} \leq AH^kM_k, \quad k=0,1,2,\dots\,.$ 
	\item[(M.2)] $M_{2k} \leq AH^{2k}M_k^2, \quad k=0,1,2,\dots\,.$
	\item[(M.3)] $\exists \ell, C>0$ such that $k! \leq C\ell^k M_k,$ for all $k \in \N_0$.
	\item[(M.4)] $\dfrac{M_r}{r!} \dfrac{M_s}{s!} \leq \dfrac{M_{r+s}}{(r+s)!}, \quad \forall r,s \in \N_0.$
\end{description}

We will assume also the logarithmic convexity:
\begin{description}
	\item[(LC)]
	$
	M_k^2 \leq M_{k-1}M_{k+1}, \quad k=1,2,3,\dots .
	$
\end{description}

Given any sequence $\{M_k\}$ that satisfies (M.0)--(M.3), there exists an alternative sequence that satisfies the logarithmic convexity and defines the same classes that we will study. So assuming (LC) does not restrict the generality compared to (M.0)--(M.3). The condition (M.4) is used only twice in this paper, in \eqref{1st-time} and \eqref{2nd-time}, to prove that an automorphism is well-defined.

From (M.0) and (LC) we have $M_k \leq M_{k+1}$, for all $k\in \N$, that is, $\{M_k\}$ is a non-decreasing sequence. Moreover, for $k \leq n$ holds
$$
M_k \cdot M_{n-k} \leq M_n.
$$

The condition (M.2) is equivalent to
$
M_k \leq AH^k \min\limits_{0\leq q \leq k} M_qM_{k-q}, 
$ (see \cite{PV84}, Lemma 5.3).

Given a sequence $\{M_k \}$ we define the associated function as
\begin{equation}\label{associated}
M(r):= \sup_{k\in \N_0} \log \frac{r^k}{M_k}, \quad r>0,
\end{equation}
and $M(0):=0$. Notice that $M$ is a non-decreasing function and by its definition,  for every $r>0$ we have
\begin{equation}\label{inf}
\exp\{M(r)\} = \sup_{k\in \N_0} \frac{r^k}{M_k}
\quad \mbox{and} \quad
\exp\{-M(r)\} = \inf_{k\in \N_0} \frac{M_k}{r^k}.
\end{equation}

It follows from these properties that for a compact Lie group $G$, for every $p,q, \delta>0$ there exists $C>0$ such that
\begin{equation}\label{propM1}
\jp{\phi}^p \exp\{-\delta M(q\jp{\phi})\} \leq C,
\end{equation}
for all $[\phi] \in \widehat{G}$. Moreover, for every $q>0$ we have
\begin{equation}\label{komine}
\exp\left\{-\tfrac{1}{2} M\left(q\jp{\phi}\right) \right\} \leq \sqrt{A}\exp\{-M\left(q_2 \jp{\phi}\right) \},
\end{equation}
for all $[\phi] \in \widehat{G}$, where $q_2= \dfrac{q}{H}$ (see \cite{DR14} for more details).

\begin{defi}
	The Komatsu class of Roumieu type $\Gamma_{\{M_k\}}(G)$ is the space of all complex-valued $C^{\infty}$ functions $f$ on $G$ such that there exist $h>0$ and $C>0$ satisfying
	$$
	\| \partial^\alpha f\|_{L^2(G)} \leq Ch^{|\alpha|}M_{|\alpha|},  \quad \forall \alpha \in \N_0^d.
	$$
\end{defi}

In the definition above, we could take the $L^\infty$-norm and obtain the same space. The elements of $\Gamma_{\{M_k\}}(G)$ are often called ultradifferentiable functions and can be characterized by their Fourier coefficients as follows:
\begin{align}
\label{roum}f \in \Gamma_{\{M_k\}}(G) \iff \exists N>0, \ \exists C>0; \quad
|\doublehat{\,f\,}\!(\xi,\eta)_{mn_{rs}}  | \leq C \exp\{-M(N(\jp{\xi}+\jp{\eta})) \},\\ \forall [\xi] \in \widehat{G_1}, \ [\eta] \in \widehat{G_2}, \ 1\leq m,n\leq d_\xi, \ 1 \leq r,s\leq d_\eta. \nonumber
\end{align}
Similarly, the ultradistribution of Roumieu type can be characterized in the following way:
\begin{align}
\label{distrou}u \in \Gamma'_{\{M_k\}}(G) \iff \forall N>0,\  \exists C_N>0; \quad
|\doublehat{\,u\,}\!(\xi,\eta)_{mn_{rs}}  | \leq C_N \exp\{M(N(\jp{\xi}+\jp{\eta})) \},\\ \forall [\xi] \in \widehat{G_1}, \ [\eta] \in \widehat{G_2}, \ 1\leq m,n\leq d_\xi, \ 1 \leq r,s\leq d_\eta. \nonumber
\end{align}	 
Next, to define Komatsu classes of Beurling type, let us replace (M.3) by the following stronger condition:
\begin{description}
	\item[(M.3')] $\forall \ell>0$, $\exists C_\ell$ such that $k! \leq C_\ell \ell^k M_k,$ for all $k \in \N_0$.
\end{description}

\begin{defi}
	The  Komatsu class of Beurling type $\Gamma_{(M_k)}(G)$ is the space of $C^{\infty}$ functions $f$ on $G$ such that for every $h>0$ there exists $C_h>0$ such that we have
	$$
	\| \partial^\alpha f\|_{L^2(G)} \leq C_hh^{|\alpha|}M_{|\alpha|}, \quad   \forall \alpha \in \N_0^n.
	$$
\end{defi}

Notice that $\Gamma_{(M_k)}(G) \subset \Gamma_{\{M_k\}}(G)$.
The elements of $\Gamma_{(M_k)}(G)$ can be characterized by their Fourier coefficients as follows:
\begin{align}
\label{caracbeurling}f \in \Gamma_{(M_k)}(G) \iff \forall N>0, \ \exists C_N>0; \quad
|\doublehat{\,f\,}\!(\xi,\eta)_{mn_{rs}} | \leq C_N \exp\{-M(N(\jp{\xi}+\jp{\eta})) \},\\ \forall [\xi] \in \widehat{G_1}, \ [\eta] \in \widehat{G_2}, \ 1\leq m,n\leq d_\xi, \ 1 \leq r,s\leq d_\eta. \nonumber
\end{align}
Similarly, the ultradistribution of Beurling type can be characterized in the following way:
\begin{align}
\label{distbeur}u \in \Gamma'_{(M_k)}(G) \iff \exists N>0,\  \exists C>0; \quad
|\doublehat{\,u\,}\!(\xi,\eta)_{mn_{rs}}  | \leq C \exp\{M(N(\jp{\xi}+\jp{\eta})) \},\\ \forall [\xi] \in \widehat{G_1}, \ [\eta] \in \widehat{G_2}, \ 1\leq m,n\leq d_\xi, \ 1 \leq r,s\leq d_\eta. \nonumber
\end{align}

\section{Partial Fourier series in Komatsu classes}
In this section, we will present the characterization of ultradifferentiable functions and ultradistributions in Komatsu classes of both Roumieu and Beurling types through the analysis of the behavior of their partial Fourier series. This will allow us to study global properties of a variable coefficient operator on a product of compact Lie groups analyzing its normal form, which was completely characterized in \cite{KMR19c}. First, we present some technical results on the associated function that we will use throughout the text.
\begin{prop}\label{prop2}
	For every $r,s>0$ we have
	\begin{enumerate}[(i)]
		\item $
		\exp\{ -M(r)\} \exp\{-M(s) \} \leq \exp\left\{-M\left(\frac{r+s}{2}\right) \right\};
		$
		\item $
		\exp\{ M(r)\} \exp\{M(s) \} \leq A\exp\left\{M\left(H(r+s)\right) \right\}.
		$
	\end{enumerate}
\end{prop}

\begin{proof}
	$(i)$ Let $r,s>0$. By \eqref{inf} we obtain
	$$
	\exp\{ -M(r)\} \exp\{-M(s) \} \leq \frac{M_j}{r^j} \frac{M_\ell}{s^\ell} \leq \frac{M_{j+\ell}}{r^j s^\ell},
	$$
	for all $j,\ell \in \N_0$. Let $k \in \N_0$. Thus for $\ell = k-j$ we have
	$$
	\exp\{ M(r)\} \exp\{M(s) \}  \geq \frac{r^j s^{k-j}}{M_k},
	$$
	so
	\begin{align*}
	2^k \exp\{ M(r)\} \exp\{M(s) \} = \sum_{j=0}^k \binom{k}{j} \exp\{ M(r)\} \exp\{M(s) \} \geq  \sum_{j=0}^k \binom{k}{j} \frac{r^j s^{k-j}}{M_k} = \frac{(r+s)^k}{M_k},
	\end{align*}
	that is,
	$$
	\exp\{ -M(r)\} \exp\{-M(s) \} \leq \frac{M_k}{\left(\frac{r+s}{2}\right)^k},
	$$
	for all $ k \in \N_0$. Therefore
	$$
	\exp\{ -M(r)\} \exp\{-M(s) \} \leq \exp\left\{-M\left(\tfrac{r+s}{2}\right) \right\}.
	$$
	
	$(ii)$ Let $r,s>0$. We have $M_{k+\ell} \leq AH^{k+\ell}M_kM_\ell$ and $r^ks^\ell \leq (r+s)^{k+\ell}$, for all $k,\ell \in \N_0$. Thus
	\begin{align*}
	\log \frac{r^k}{M_k} + \log\frac{s^\ell}{M_\ell} = \log \frac{r^ks^\ell}{M_k M_\ell} \leq \log A \frac{H(r+s)^{k+\ell}}{M_{k+\ell}} &= \log A + \log \frac{(H(r+s))^{k+\ell}}{M_{k+\ell}} \nonumber \\
	&\leq \log A + M(H(r+s)). \nonumber
	\end{align*}
	For every $\ell \in \N_0$ fixed we have
	$$
	\log \frac{r^k}{M_k} \leq  \log A + M(H(r+s)) - \log\frac{s^\ell}{M_\ell} \implies M(r) \leq \log A + M(H(r+s)) - \log\frac{s^\ell}{M_\ell}.
	$$
	Now,
	$$
	\log\frac{s^\ell}{M_\ell} \leq \log A + M(H(r+s)) - M(r), \quad \forall \ell \in \N_0,
	$$
	which implies that
	$$
	M(s) \leq \log A + M(H(r+s)) - M(r).
	$$ 
	By the properties of the exponential function we obtain
	$$	
	\exp\{ M(r)\} \exp\{M(s) \} \leq A\exp\left\{M\left(H(r+s)\right) \right\},
	$$
	and the proof is complete.
\end{proof}
\begin{prop}\label{propM3}
	For every $r,s>0$ and $t\in\N_0$ we have 
	\begin{enumerate}[(i)]
		\item
		$	
		r^t\exp\{-M(sr) \} \leq A \left(Hs^{-1}\right)^{t} M_t \exp\{-M(H^{-1}sr) \};
		$
		\item
		$	
		r^t\exp\{M(sr) \} \leq A s^{-t} M_t \exp\{M(Hsr) \}.
		$
	\end{enumerate}
\end{prop}
\begin{proof}
	$(i)$ Let $r,s,t >0 $. We have
	$$
	r^t \exp\{-M(sr) \} \leq r^t\frac{M_k}{s^k r^k} = s^{-t} \frac{M_k}{(sr)^{k-t}}, \quad \forall k \geq t.
	$$
	Since $M_k \leq AH^kM_t M_{k-t}$, for all $k \geq t$, we obtain
	$$
	r^t \exp\{-M(sr) \} \leq As^{-t}H^kM_t \frac{M_{k-t}}{(sr)^{k-t}} = A(s^{-1}H)^t M_t \frac{M_{k-t}}{(H^{-1}sr)^{k-t}}, \quad \forall k \geq t,
	$$
	Therefore 
	$$
	r^t\exp\{-M(sr) \} \leq A \left(Hs^{-1}\right)^{t} M_t \exp\{-M(H^{-1}sr) \}.
	$$
	
	$(ii)$ Let $r,s,t>0$. We have
	$$
	r^t \exp\{M(sr) \}  = r^t \sup_{k\in\N_0} \frac{(sr)^k}{M_k} = \sup_{k\in \N_0} \frac{s^{k} r^{k+t}}{M_k} = s^{-t} \sup_{k \in \N_0} \frac{(sr)^{k+t}}{M_{k}} 
	$$
	Since $M_{k+t} \leq AH^{k+t} M_kM_t$, we obtain
	$$
	r^t \exp\{M(sr) \}  \leq As^{-t}M_t \sup_{k\in \N_0} \frac{(Hsr)^{k+t}}{M_{k+t}} \leq As^{-t}M_t \sup_{\ell \in \N_0} \frac{(Hsr)^\ell}{M_\ell}.
	$$
	Therefore 
	$$	
	r^t\exp\{M(sr) \} \leq A s^{-t} M_t \exp\{M(Hsr) \}.
	$$
\end{proof}
\begin{thm}\label{partialroumieu}
	Let $G_1$ and $G_2$ be compact Lie groups, set $G=G_1\times G_2$, and let $f \in C^\infty(G)$. Then $f \in \Gamma_{\{M_k\}}(G)$ if and only if $\widehat{f}(\: \cdot\:, \eta)_{rs} \in \Gamma_{\{M_k\}}(G_1)$ for every $[\eta] \in \widehat{G_2}$, $1 \leq r,s \leq d_\eta$ and there exist $h,C,\varepsilon>0$ such that
\begin{equation}\label{hyporoumieupart}
	\max_{x_1 \in G_1} |\partial^\alpha \widehat{f}(x_1,\eta)_{rs}| \leq Ch^{|\alpha|}  M_{|\alpha|} \exp\{-M(\varepsilon\jp{\eta}) \},
\end{equation}
	for all $ [\eta] \in \widehat{G_2}, \ 1 \leq r,s \leq d_\eta$  and  $\alpha \in \N_0^{d_1}.$
\end{thm}

\begin{proof}
	
	$(\impliedby)$ Let $\alpha \in \N_0$. Recall that $-\nu_{[\xi]}$ is the eigenvalue of the Laplacian operator $\mathcal{L}_{G_1}$ associated to the eigenfunctions $\{\xi_{mn}, \ 1 \leq m,n \leq d_\xi\}$. By \eqref{symbollaplacian}, we obtain
	\begin{align}
	\nu_{[\xi]}^\alpha |\doublehat{\,f\,}\!(\xi,\eta)_{rs_{mn}}| &= \left|\widehat{\mathcal{L}_{G_1}^\alpha{\widehat{f}}}(\xi,\eta)_{rs_{mn}} \right|\nonumber \\ &= \left|\int_{G_1} \mathcal{L}_{G_1}^\alpha\widehat{f}(x_1,\eta)_{rs} \overline{\xi(x_1)_{nm}} \, dx_1 \right| \nonumber \\
	&\leq \int_{G_1}  |\mathcal{L}_{G_1}^\alpha\widehat{f}(x_1,\eta)_{rs}| |{\xi(x_1)_{nm}}|\, dx_1\nonumber \\
	&\leq \left(\int_{G_1}  |\mathcal{L}_{G_1}^\alpha\widehat{f}(x_1,\eta)_{rs}|^2\, dx_1\right)^{1/2} \left(\int_G |{\xi(x_1)_{nm}}|^2\, dx_1 \right)^{1/2}\nonumber.
	\end{align}
	Notice that, by \eqref{ortho}, we have $\|\xi_{nm}\|_{L^2(G_1)} \leq 1$, for all $[\xi]\in \widehat{G_1}$. Moreover, we can write $\mathcal{L}_{G_1}^\alpha$ as a sum of $d_1^\alpha$ derivatives of order $2\alpha$,	where $d_1=\dim G_1$. So, by \eqref{hyporoumieupart}, we obtain 
	$$
	\nu_{[\xi]}^\alpha |\doublehat{\,f\,}\!(\xi,\eta)_{rs_{mn}}| \leq Cd_1^\alpha h^{2\alpha}M_{2\alpha}\exp\{-M(\varepsilon\jp{\eta})\}. 
	$$
	By definition of $\jp{\xi}$, there exists $C>0$ such that $\jp{\xi}^2\leq C \nu_{[\xi]}$,
	for all non-trivial representations. By the property (M.2) of the sequence $\{ M_k\}$, we have $M_{2\alpha} \leq AH^{2\alpha} M_\alpha^2$. Thus
	$$
	|\doublehat{\,f\,}\!(\xi,\eta)_{rs_{mn}}| \leq C(\sqrt{d_1}hH)^{2\alpha} \jp{\xi}^{-2\alpha}  M_\alpha^2 \exp\{-M(\varepsilon\jp{\eta}) \}, \quad \forall \alpha \in \N_0 .
	$$
	Hence, 
	\begin{align}
	|\doublehat{\,f\,}\!(\xi,\eta)_{rs_{mn}}| &\leq C \left(\inf_{\alpha\in \N_0} \frac{M_\alpha}{(\jp{\xi}(\sqrt{d_1}hH)^{-1})^\alpha} \right)^2 \exp\{-M(\varepsilon\jp{\eta}) \nonumber \\
	&= C \exp\{-2M((\sqrt{d_1}hH)^{-1}\jp{\xi})\}  \exp\{-M(\varepsilon\jp{\eta}) \} \nonumber\\
	&\leq C \exp\{-M((\sqrt{d_1}hH)^{-1}\jp{\xi})\}  \exp\{-M(\varepsilon\jp{\eta}) \} \nonumber.
	\end{align}
	Set $2N=\min\{(\sqrt{d_1}hH)^{-1}, \varepsilon \}$. In this way, we get
	$$
	|\doublehat{\,f\,}\!(\xi,\eta)_{rs_{mn}}| \leq C \exp\{-M(2N\jp{\xi})\}  \exp\{-M(2N\jp{\eta}) \}
	$$
	and by Proposition \ref{prop2},
	$$
	|\doublehat{\,f\,}\!(\xi,\eta)_{rs_{mn}}| \leq C \exp\{-M({N}(\jp{\xi}+\jp{\eta}))\}, 
	$$
	for all $[\xi] \in \widehat{G_1}$ non-trivial, $[\eta] \in \widehat{G_2}$. It is easy to see that we can also  obtain this inequality for the trivial representation of $G_2$ from the hypothesis. 
	Therefore $f \in \Gamma_{\{M_k\}}(G)$.
	
	$(\implies)$ We can characterize the elements of $\Gamma_{\{M_k\}}(G)$ as follows (Theorem 2.3 of \cite{DR16}):
	$\varphi \in \Gamma_{\{M_k\}}(G)$ if and only if there exist $C,h>0$ such that
	$$
	\max_{(x_1,x_2)\in G}|\partial_1^\alpha \partial_2^\beta \varphi(x_1,x_2)| \leq Ch^{|\alpha|+|\beta|}M_{|\alpha|+|\beta|},
	$$
	for all $\alpha \in \N_0^{d_1}, \beta \in \N_0^{d_2}$.
	
	For $f \in \Gamma_{\{M_k\}}(G)$ we have
	\begin{align}
	\nu_{[\eta]}^\beta|\partial_1^\alpha{\widehat{f}}(x_1,\eta)_{rs}| &= |\partial_1^\alpha\widehat{ \mathcal{L}_{G_2}^\beta f}(x_1,\eta)_{rs}| \nonumber \\
	&\leq \int_{G_2}|\partial_1^\alpha\mathcal{L}_{G_2}^\beta  f(x_1,x_2)| |\overline{\eta(x_2)_{sr}}| \, dx_2 \nonumber \\
	&\leq \left(\int_{G_2} |\partial_1^\alpha\mathcal{L}_{G_2}^\beta  f(x_1,x_2)|^2 \, dx_2\right)^{1/2} \left(\int_{G_2} |\eta(x_2)_{sr}|^2 \, dx_2 \right)^{1/2} \nonumber  \\
	&\leq \frac{1}{\sqrt{d_\eta}}\sum_{|\gamma|=2\beta} \max_{(x_1,x_2) \in G}|\partial_1^\alpha \partial_2^\gamma f(x_1,x_2)| \nonumber \\
	&\leq Cd_2^\beta h^{|\alpha|+2\beta}M_{|\alpha|+2\beta}, \nonumber 
	\end{align}
	where $d_2=\dim G_2$. Thus, when $[\eta]$ is not trivial we obtain
	\begin{align}
	|\partial_1^\alpha{\widehat{f}}(x_1,\eta)_{rs}|&\leq Cd_2^\beta h^{|\alpha|+2\beta}M_{|\alpha|+2\beta} \jp{\eta}^{-2\beta} \nonumber \\
	&\leq  C h^{|\alpha|+2\beta} AH^{|\alpha|+2\beta}M_{|\alpha|} h^{2\beta} d_2^\beta M_{2\beta} \jp{\eta}^{-2\beta} \nonumber \\
	&\leq C(hH)^{|\alpha|}M_{|\alpha|}h^{2\beta} d_2^\beta H^{4\beta}M_{\beta}^2 \jp{\eta}^{-2\beta} \nonumber \\
	&\leq C(hH)^{|\alpha|} M_{|\alpha|} \exp\{-2M((\sqrt{d_2}hH^2)^{-1}\jp{\eta})\} \nonumber \\
	&\leq C(hH)^{|\alpha|} M_{|\alpha|} \exp\{-M((\sqrt{d_2}hH^2)^{-1}\jp{\eta})\} \nonumber.
	\end{align}
	Put $h'= hH$ and $\varepsilon = (\sqrt{d_2}hH^2)^{-1}$ to obtain
	$$
	\max_{x_1 \in G_1} |\partial_1^\alpha \widehat{f}(x_1,\eta)_{rs}| \leq C{h'}^{|\alpha|}  M_{|\alpha|} \exp\{-M(\varepsilon\jp{\eta}) \}, 
	$$
for all non-trivial $[\eta] \in \widehat{G_2}, \ 1 \leq r,s \leq d_\eta, \  \alpha \in \N_0^n$.

	For $[\eta] = [\mathds{1}_{G_2}]$ we have
	\begin{align}
	|\partial_1^\alpha \widehat{f}(x_1,\mathds{1}_{G_2})| &= \left| \int_{G_2} \partial_1^\alpha f(x_1,x_2) \, dx_2\right| \nonumber \\
	&\leq |\partial_1^\alpha f(x_1,x_2)| \nonumber \\
	&\leq Ch^{|\alpha|}M_{|\alpha|} \nonumber.
	\end{align}
	In this way, adjusting $C$ if necessary, we obtain
	$$
	|\partial_1^\alpha \widehat{f}(x_1,\mathds{1}_{G_2})|  \leq Ch^{|\alpha|}M_{|\alpha|} \exp\{-M(\varepsilon\jp{\mathds{1}_{G_2}}) \},
	$$
	and so the proof is complete.
\end{proof}

\begin{thm}\label{partialbeurling}
	Let $G_1$ and $G_2$ be compact Lie groups, set $G=G_1\times G_2$, and let $f \in C^\infty(G)$. Then $f \in \Gamma_{(M_k)}(G)$ if and only if $\widehat{f}(\:\cdot\:, \eta)_{rs} \in \Gamma_{(M_k)}(G_1)$ for every $[\eta] \in \widehat{G_2}$, $1 \leq r,s \leq d_\eta$ and for all $h>0$ and $\varepsilon>0$ there exists $C_{h\varepsilon}>0$ such that we have
	$$
	\max_{x_1 \in G_1} |\partial^\alpha \widehat{f}(x_1,\eta)_{rs}| \leq C_{h\varepsilon}h^{|\alpha|}  M_{|\alpha|} \exp\{-M(\varepsilon\jp{\eta}) \},
	$$
	for all $ [\eta] \in \widehat{G_2}, \ 1 \leq r,s \leq d_\eta$  and  $\alpha \in \N_0^{d_1}.$
\end{thm}
\begin{proof}
	$(\impliedby)$ By the proof of Theorem \ref{partialroumieu}, we have
	$$
	|\doublehat{\,f\,}\!(\xi,\eta)_{rs_{mn}}|  \leq C_{h\varepsilon} \exp\{-M((\sqrt{d_1}hH)^{-1}\jp{\xi})\}  \exp\{-M(\varepsilon\jp{\eta}) \} .
	$$
	Given $N>0$, choose $h=\dfrac{1}{2\sqrt{d_1} NH}$ and $\varepsilon=2N$. So
	\begin{align}
	|\doublehat{\,f\,}\!(\xi,\eta)_{rs_{mn}}|  &\leq C_N \exp\{-M(2N\jp{\xi})\}\exp\{-M(2N\jp{\eta})\} \nonumber \\
	&\leq C_N \exp\{-M(N(\jp{\xi}+\jp{\eta})) \} \nonumber .
	\end{align}
	Therefore $f \in \Gamma_{(M_k)}(G)$.
	
	$(\implies)$ We can characterize the elements of $\Gamma_{\{M_k\}}(G)$ as follows (see \cite{DR16}):
	$\varphi \in \Gamma_{\{M_k\}}(G)$ if and only if for all $h>0$ there exists $C_h>0$ such that
	$$
	\max_{(x_1,x_2)\in G}|\partial_1^\alpha \partial_2^\beta \varphi(x_1,x_2)| \leq Ch^{|\alpha|+|\beta|}M_{|\alpha|+|\beta|},
	$$
	for all $\alpha \in \N_0^{d_1}, \beta \in \N_0^{d_2}$.
	Let $f \in \Gamma_{(M_k)}$. In the proof of Theorem \ref{partialroumieu} we have obtained
	$$
	|\partial_1^\alpha{\widehat{f}}(x_1,\eta)_{rs}| \leq C_h(hH)^{|\alpha|} M_{|\alpha|} \exp\{-M((\sqrt{n}hH^2)^{-1}\jp{\eta})\} .
	$$
	Given $\ell, \varepsilon >0$. If $\ell\varepsilon < (\sqrt{n}H)^{-1}$, take $h = \ell H^{-1}$. In this case,
	\begin{align}
	|\partial_1^\alpha{\widehat{f}}(x_1,\eta)_{rs}| &\leq C_{\ell \varepsilon }\ell^{|\alpha|} M_{|\alpha|} \exp\{-M((\sqrt{n}\ell H)^{-1}\jp{\eta})\} \nonumber \\
	&\leq C_{\ell \varepsilon }\ell^{|\alpha|} M_{|\alpha|} \exp\{-M((\varepsilon\jp{\eta})\}. \nonumber
	\end{align}
	If $\ell\varepsilon \geq (\sqrt{n}H)^{-1}$, take $h=(\sqrt{n}\varepsilon H^2)^{-1}$. So
	\begin{align}
	|\partial_1^\alpha{\widehat{f}}(x_1,\eta)_{rs}| &\leq C_{\ell \varepsilon }(\sqrt{n}\varepsilon H^2)^{-|\alpha|} M_{|\alpha|} \exp\{-M((\varepsilon\jp{\eta})\} \nonumber \\
	&\leq C_{\ell \varepsilon }\ell^{|\alpha|} M_{|\alpha|} \exp\{-M((\varepsilon\jp{\eta})\} ,\nonumber 
	\end{align}
	and the proof is complete.
\end{proof}

\begin{thm}\label{ultraroumieu}
	Let $G_1$ and $G_2$ be compact Lie groups, and set $G=G_1\times G_2$. Then $u\in \Gamma'_{\{M_k\}}(G)$ if and only if for all $\varepsilon,h>0$ there exists $C_{h\varepsilon}>0$ such that
	$$
	|\langle \widehat{u}(\: \cdot\: ,\eta)_{rs}, \varphi \rangle | \leq C_{h\varepsilon} \|\varphi\|_{h}\exp\{M(\varepsilon\jp{\eta}) \}, \quad \forall \varphi \in \Gamma_{M_k}(G_1),
	$$
	where $\|\varphi\|_{h}:=\sup\limits_{\alpha, x_1}\bigl|\partial^\alpha\varphi(x_1)\bigr|h^{-|\alpha|}M_{|\alpha|}^{-1}$.
\end{thm}
\begin{proof}
	$(\impliedby)$ Let $\varphi = \overline{\xi_{nm}}$. We have $$|\partial^\beta \xi_{nm}(x_1)| \leq C C_0^{|\beta|} \jp{\xi}^{p+|\beta|},$$
	where $p$ is any natural number satisfying $p \geq \frac{\dim G}{2}$ (see \cite{DR14}). Then
	\begin{align*}
| \jp{\widehat{u}(\: \cdot \:, \eta)_{rs}, \overline{\xi_{nm}}}| &\leq C_{h\varepsilon}\|\overline{\xi_{nm}}\|_{h} \exp\{M(\varepsilon\jp{\eta}) \} \nonumber \\
&= C_{h\varepsilon}\sup_{\alpha, x_1} |\partial^\alpha \overline{\xi_{nm}}(x_1) h^{-|\alpha|} M_{|\alpha|}^{-1} | \exp\{M(\varepsilon\jp{\eta}) \} \nonumber \\
	&\leq C_{h\varepsilon} \jp{\xi}^p \sup_{\alpha} | C_0^{|\alpha|}\jp{\xi}^{|\alpha|} h^{-|\alpha|} M_{|\alpha|}^{-1}| \exp\{M(\varepsilon\jp{\eta}) \} \nonumber \\
	&= C_{h\varepsilon} \jp{\xi}^p \exp\{M(h^{-1}C_0 \jp{\xi}) \} \exp\{M(\varepsilon\jp{\eta}) \} \nonumber.
	\end{align*}
	
	By Proposition \ref{propM3}, we have $$\jp{\xi}^p \exp\{M(h^{-1}C_0 \jp{\xi}) \} \leq A(h^{-1}C_0)^{-p}M_p\exp\{M(Hh^{-1}C_0 \jp{\xi}) \}.$$ 
	
	By Proposition \ref{prop2}, we obtain
	$$
	| \jp{\widehat{u}(\: \cdot \:, \eta)_{rs}, \overline{\xi_{nm}}}|  \leq C_{h\varepsilon} \exp\{M(H( Hh^{-1}C_0 \jp{\xi} + \varepsilon\jp{\eta}) ) \}.
	$$
	
	Given $N>0$, choose $h = \frac{H^2C_0}{N}$ and $\varepsilon=\frac{N}{H}$. In this way,
	$$
	|\doublehat{\,u\,}\!(\xi,\eta)_{mn_{rs}}| \leq C_N \exp\{M(N(\jp{\xi}+\jp{\eta})) \},
	$$
	which implies that $u \in \Gamma_{\{M_k\}}'(G)$.
	
	$(\implies)$ Since $ u \in \Gamma_{\{M_k\}}'(G)$, for every $\ell>0$, there exists $C_\ell>0$ such that
	$$
	|\jp{u,\psi}| \leq C_\ell \sup_{\alpha, \beta} \ell^{|\alpha|+|\beta|} M_{|\alpha|+|\beta|}^{-1} ||\partial^\alpha_1\partial^\beta_2 \psi||_{L^\infty(G)},
	$$
	for all $\psi \in \Gamma_{\{M_k\}}(G)$. Given $\varphi \in \Gamma_{\{M_k\}}(G_1)$, take $\psi=\varphi \times \overline{\eta_{sr}}$. Then
	\begin{align*}
	|\jp{\widehat{u}(\: \cdot \:, \eta)_{rs},\varphi}| &= |\jp{u,\varphi \times \overline{\eta_{sr}}}|\\ &\leq C_\ell \sup_{\alpha, \beta} \ell^{|\alpha|+|\beta|} M_{|\alpha|+|\beta|}^{-1} \sup_{x_1} |\partial_1^\alpha \varphi(x_1)| \sup_{x_2} |\partial^\beta_2 \overline{\eta_{sr} (x_2)}| .\nonumber
	\end{align*}
	Similar to what was done above, we have
	$$
	\sup_{\beta, x_2} |\partial^\beta_2 \overline{\eta_{sr} (x_2)} \ell^{|\beta|} M_{|\beta|}^{-1}  | \leq C_\ell \exp\{M(H\ell C_0\jp{\eta}) \}.
	$$
	By the property $M_{|\alpha|} M_{|\beta|} \leq M_{|\alpha|+|\beta|}$ we obtain
	$$ 
	|\langle \widehat{u}(\: \cdot\: ,\eta)_{rs}, \varphi \rangle | \leq C_\ell \sup_{\alpha, x_1} |\partial^\alpha_1 \varphi(x_1) \ell^{|\alpha|} M_{|\alpha|}^{-1}| \exp\{M(H\ell C_0\jp{\eta}) \}.
	$$
	Given $h,\varepsilon>0$. If $\varepsilon h \leq C_0H$, take $\ell = \frac{\varepsilon}{C_0 H}$. Thus $\ell \leq h^{-1}$ and
	$$
	|\langle \widehat{u}(\: \cdot\: ,\eta)_{rs}, \varphi \rangle | \leq C_{h\varepsilon} \|\varphi\|_{h}\exp\{M(\varepsilon\jp{\eta}) \}.
	$$
	On the other hand, if $\varepsilon h > C_0H$, take $\ell = h^{-1}$. Thus $H\ell C_0 < \varepsilon$ and
	$$
	|\langle \widehat{u}(\: \cdot\: ,\eta)_{rs}, \varphi \rangle | \leq C_{h\varepsilon} \|\varphi\|_{h}\exp\{M(\varepsilon\jp{\eta}) \},
	$$
	completing the proof.
\end{proof}

\begin{thm}\label{ultrabeurling}
	Let $G_1$ and $G_2$ be compact Lie groups, and set $G=G_1\times G_2$ . Then  $u\in \Gamma'_{(M_k)}(G)$ if and only if there exist $\varepsilon,h,C>0$ such that we have
	$$
	|\langle \widehat{u}(\:\cdot\:,\eta)_{rs}, \varphi \rangle | \leq C \|\varphi\|_{h}\exp\{M(\varepsilon\jp{\eta}) \}, \quad \forall \varphi \in \Gamma_{(M_k)}(G_1).
	$$
\end{thm}

	The proof of this theorem is analogous to the Roumieu case and it will be omitted.

\section{Normal Form}\label{secnormal}

Let $G_1$  and $G_2$ be compact Lie groups and consider the operator $L_a$ defined on $G:=G_1\times G_2$ by
\begin{equation}\label{L-a}
L_a=X_1+a(x_1)X_2,
\end{equation}
where $X_1  \in \mathfrak{g}_1$, $X_2 \in \mathfrak{g}_2$, and $a \in \Gamma_{\{M_k\}}(G_1)$ is a real-valued function. For each $[\xi] \in \widehat{G_1}$, we can choose a representative $\xi \in \mbox{Rep}({G_1})$ such that
\begin{equation*}\label{symbol1}
\sigma_{X_1}(\xi)_{mn} =i\lambda_m(\xi) \delta_{mn}, \quad 1 \leq m,n \leq d_\xi,
\end{equation*}
where  $\lambda_m(\xi)\in\R$  for all $[\xi] \in\widehat{G_1}$ and $1 \leq m \leq d_\xi$.
Similarly, for each $[\eta] \in \widehat{G_2}$, we can choose a representative $\eta \in \mbox{Rep}({G_2})$ such that
\begin{equation*}\label{symbol2}
\sigma_{X_2}(\eta)_{rs} =i\mu_r(\eta) \delta_{rs}, \quad 1 \leq r,s \leq d_\eta,
\end{equation*}
where  $\mu_r(\eta)\in\R$  for all $[\eta] \in\widehat{G_2}$ and $1 \leq r \leq d_\eta$.

The idea is to apply the same technique used in \cite{BP99,Ber99,KMR19b} and several other references of studying the global properties of \eqref{L-a} by analyzing the same properties of the equivalent constant-coefficient  operator $L_{a_0} = X_1 + a_0 X_2$, where
$$a_0:=  \displaystyle\int_{G_1} a(x_1)\, dx_1.$$

For this end, we have the following additional hypothesis:

\begin{hyp}
For the real-valued function $a \in \Gamma_{\{M_k\}}(G_1)$ (respectively, $a \in \Gamma_{(M_k)}(G_1)$), there exists $A \in \Gamma_{\{M_k\}}(G_1)$ (respectively, $A \in \Gamma_{(M_k)}(G_1)$) such that 
\begin{equation}\label{A-function}
X_1 A(x_1) = a(x_1) - a_0,
\end{equation} 
for all $x_1 \in G_1$.
\end{hyp}

\begin{rem}\label{assumption}
	When $G_1$ is the one-dimensional torus, the operator $X_1=\partial_t$ is globally solvable and $a - a_0$ belongs to the set of admissible functions, therefore this hypothesis is satisfied. However, for other compact Lie groups, including higher-dimensional torus and the sphere $\St$, it is not difficult to construct examples of a function $a$ for which there is no $A$ satisfying \eqref{A-function}.  
\end{rem}

Now we define the operator $\Psi_a$ as
\begin{equation}\label{psialpha2}
\Psi_a   u(x_1,x_2) := \sum_{[\eta]\in \widehat{G_2}}d_\eta \sum_{r,s = 1}^{d_\eta} e^{i\mu_r(\eta)A(x_1)}\widehat{u}(x_1,\eta)_{rs}\,{\eta_{sr}(x_2)}.
\end{equation}

In \cite{KMR19b} it was proved that $\Psi_a$ is an automorphism of $C^\infty(G)$ and $\mathcal{D}'(G)$, with inverse $\Psi_{-a}$. Moreover, we have
\begin{equation}\label{conjugation}
\Psi_a\circ L_a = L_{a_0}\circ\Psi_a.
\end{equation} 

Since the operator $L_a$ is the same as in \cite{KMR19b}, the expression \eqref{conjugation} remains valid in Komatsu classes.

In the next results, we present sufficient conditions for the operator $\Psi_a$ to be an automorphism in the space of ultradifferentiable functions and ultradistributions of both Roumieu and Beurling types.
First, by the definition of ultradifferentiable functions, there exist $K',\ell'>0$ such that for all $\alpha \in \N_0^{d_1}$ we have
$$
|\partial^\alpha A(x_1)| \leq K' \ell'^{|\alpha|}M_{|\alpha| }, \quad \forall x_1 \in G_1.
$$

Since $M_{|\alpha|} \leq AH^{|\alpha|} M_1 M_{|\alpha|-1}$, we obtain for all non-zero $\alpha \in \N_0^{d_1}$ 
\begin{equation}\label{propertyA}
|\partial^\alpha A(x_1)| \leq K \ell^{|\alpha|-1}M_{|\alpha| -1}, \quad \forall x_1 \in G_1,
\end{equation}
where $K=K'\ell' HAM_1$ and $\ell = \ell'H$.

Similarly, if $A \in \Gamma_{(M_k)}(G_1)$, for any $\ell>0$ there exists $K_\ell>0$ such that for all non-zero $\alpha \in \N_0^{d_1}$ we have
\begin{equation}\label{propertyA2}
|\partial^\alpha A(x_1)| \leq K_\ell \ell^{|\alpha|-1}M_{|\alpha| -1}, \quad \forall x_1 \in G_1.
\end{equation}

\begin{prop}\label{automrou}
	Let $a\in \Gamma_{\{M_k\}}(G_1)$. Then the operator $\Psi_a$, defined in \eqref{psialpha2}, is an automorphism of $\Gamma_{\{M_k\}}(G_1\times G_2)$.
\end{prop}
\begin{proof}
It is enough to show that $\Psi_a   u \in \Gamma_{\{M_k\}}(G_1 \times G_2)$ when $u \in \Gamma_{\{M_k\}}(G_1 \times G_2)$.
By the characterization of ultradifferentiable functions of Roumieu type from their partial Fourier coefficients, there exist $C,h,\varepsilon>0$ such that
\begin{equation}\label{propu}
|\partial^\alpha \widehat{u}(x_1,\eta)_{rs}| \leq Ch^{|\alpha|}M_{|\alpha|} \exp\{-M(\varepsilon\jp{\eta}) \},
\end{equation}
for all $\alpha \in \N_0^{d_1}$, $x_1\in G_1$, $[\eta] \in \widehat{G_2}$ and $1\leq r,s \leq d_\eta$. Notice that
$$
\widehat{\Psi_{a}  u}(x_1,\eta)_{rs} = e^{i\mu_r(\eta)A(x_1)}\widehat{u}(x_1,\eta)_{rs}.
$$
Thus, for $\alpha \in \N_0^{d_1}$ we have
\begin{align*}
|\partial^\alpha \widehat{\Psi_{a}  u}(x_1,\eta)_{rs}| &= \left|\partial^\alpha \left(e^{i\mu_r(\eta)A(x_1)}\widehat{u}(x_1,\eta)_{rs}\right)\right| \leq \sum_{\beta \leq \alpha} \binom{\alpha}{\beta} \left|\partial^\beta e^{i\mu_r(\eta)A(x_1)}\right| \left|\partial^{\alpha-\beta }\widehat{u}(x_1,\eta)_{rs}\right|.
\end{align*}
Using that $|\mu_r(\eta)| \leq \jp{\eta}$ and \eqref{propertyA}, we have by Fa\`{a} di Bruno's Formula that
$$
|\partial^\beta e^{i\mu_r(\eta)A(x_1)}| \leq \sum_{k=1}^{|\beta|} K^k\jp{\eta}^k \ell^{|\beta|-k} \left( \sum_{\lambda \in \Delta({|\beta|},k)}\binom{{|\beta|}}{\lambda} \frac{1}{r(\lambda)!}\prod_{j=1}^k M_{\lambda_j-1} \right),
$$
where $\Delta({|\beta|},k) = \{\lambda\in\N^k; |\lambda|={|\beta|} \mbox{ and } \lambda_1\geq \cdots \geq \lambda_k \geq 1 \}$ and $r(\lambda) \in \N_0^{d_1}$, where $r(\lambda)_j$ counts how many times $j$ appears on $\lambda$. 

By property (M.4) of the sequence $\{M_k\}_{k\in\N_0}$ we obtain
\begin{equation}\label{1st-time}
\binom{|\beta|}{\lambda}\prod_{j=1}^{|\beta|} M_{\lambda_j-1}= |\beta|! \prod_{j=1}^{|\beta|}\frac{M_{\lambda_j-1}}{\lambda_j!}\leq |\beta|! \prod_{j=1}^{|\beta|}\frac{M_{\lambda_j-1}}{(\lambda_j-1)!} \leq |\beta|! \frac{M_{|\beta|-k}}{(|\beta|-k)!},
\end{equation}
for $\lambda \in \Delta(|\beta|,k)$. Using the fact that $$\sum\limits_{\lambda \in \Delta({|\beta|},k)} \frac{1}{r(\lambda)!} = \binom{|\beta|-1}{k-1}\frac{1}{k!},$$ we have
\begin{equation}\label{propexp}
|\partial^\beta e^{i\mu_r(\eta)A(x_1)}| \leq \sum_{k=1}^{|\beta|} \binom{|\beta|-1}{k-1}\frac{1}{k!} K^k\jp{\eta}^k \ell^{|\beta|-k} |\beta|! \frac{M_{|\beta|-k}}{(|\beta|-k)!}.
\end{equation}
By \eqref{propu}, we have
\begin{align*}
|\partial^\alpha \widehat{\Psi_{a}  u}(x_1,\eta)_{rs}| &\leq C\sum_{\beta \leq \alpha} \binom{\alpha}{\beta}\sum_{k=1}^{|\beta|} \binom{|\beta|-1}{k-1}\frac{1}{k!} K^k\jp{\eta}^k \ell^{|\beta|-k} \\ &\quad \times|\beta|! \frac{M_{|\beta|-k}}{(|\beta|-k)!}  h^{|\alpha|-|\beta|} M_{|\alpha|-|\beta|} \exp\{-M(\varepsilon\jp{\eta}) \}.
\end{align*}
By Proposition \ref{propM3},
$$
\jp{\eta}^{k }\exp\{-M(\varepsilon\jp{\eta}) \} \leq A\left(\frac{H}{\varepsilon}\right)^{k } M_{k }\exp\{-M(\varepsilon H^{-1} \jp{\eta}) \}.
$$
So,
\begin{align*}
	|\partial^\alpha \widehat{\Psi_{a}  u}(x_1,\eta)_{rs}| &\leq AC\sum_{\beta \leq \alpha} \binom{\alpha}{\beta}\sum_{k=1}^{|\beta|} \binom{|\beta|-1}{k-1}\left(\frac{KH}{\ell \varepsilon}\right)^{k} \ell^{|\beta|}  h^{|\alpha|-|\beta|}\\ &\quad \times  |\beta|! \frac{M_{|\beta|-k}}{(|\beta|-k)!} \frac{M_k}{k!} M_{|\alpha|-|\beta|} \exp\{-M(\varepsilon\jp{\eta}) \}.
\end{align*}
Notice that
$$
 |\beta|! \frac{M_{|\beta|-k}}{(|\beta|-k)!} \frac{M_k}{k!} M_{|\alpha|-|\beta|} \leq  |\beta|! \frac{M_{|\beta|}}{|\beta|!}M_{|\alpha|-|\beta|} \leq M_{|\alpha|}.
$$ 
Denote by $S=\max\{\frac{KH}{\varepsilon},\ell \}$. Thus
$$
|\partial^\alpha \widehat{\Psi_{a}  u}(x_1,\eta)_{rs}| \leq AC\sum_{\beta \leq \alpha} \binom{\alpha}{\beta}S^{|\beta|}h^{|\alpha|-|\beta|}M_{|\alpha|} \exp\{-M(\varepsilon H^{-1} \jp{\eta}) \} \sum_{k=1}^{|\beta|} \binom{|\beta|-1}{k-1}.
$$
We have $ \sum\limits_{k=1}^{|\beta|} \binom{|\beta|-1}{k-1} = 2^{|\beta|-1}$. Moreover, 
\begin{equation}\label{binom}
\sum_{\beta \leq \alpha} \binom{\alpha}{\beta}(2S)^{|\beta|}h^{|\alpha|-|\beta|}= \sum_{|\beta|=0} ^{|\alpha|} \binom{|\alpha|}{|\beta|}(2S)^{|\beta|}h^{|\alpha|-|\beta|} = (2S+h)^{|\alpha|}.
\end{equation}
In this way
$$
|\partial^\alpha \widehat{\Psi_{a}  u}(x_1,\eta)_{rs}| \leq AC\left(2S+h \right)^{|\alpha|} M_{|\alpha|} \exp\{-M(\varepsilon H^{-1} \jp{\eta}) \}.
$$
By Theorem \ref{partialroumieu} we conclude that $\Psi_a   u \in \Gamma_{\{M_k\}}(G_1\times G_2)$.
\end{proof}

\begin{prop}
	Assume that $a \in \Gamma_{(M_k)}(G_1)$. Then $\Psi_a$ is an automorphism of $\Gamma_{(M_k)}(G_1 \times G_2)$.
\end{prop}
\begin{proof}
Let $u \in \Gamma_{(M_k)}(G_1 \times G_2)$. By \eqref{propertyA} we have that
$$
|\partial^\alpha A(x_1)| \leq K_\ell \ell^{|\alpha|-1}M_{|\alpha| -1}, \quad \forall x_1 \in G_1.
$$
By Theorem \ref{partialbeurling} for all $h, \varepsilon >0$ there exists $C_{h\varepsilon}>0$ such that
\begin{equation}\label{propub2}
|\partial^\alpha \widehat{u}(x_1,\eta)_{rs}| \leq C_{h\varepsilon}h^{|\alpha|}M_{|\alpha|} \exp\{-M(\varepsilon\jp{\eta}) \},
\end{equation}
for all $\alpha \in \N_0^{d_1}$, $x_1\in G_1$, $[\eta] \in \widehat{G_2}$ and $1\leq r,s \leq d_\eta$. We can follow the proof of  Roumieu type case and obtain
$$
|\partial^\alpha \widehat{\Psi_{a}  u}(x_1,\eta)_{rs}| \leq C_{h\varepsilon}\left(2S+h \right)^{|\alpha|} M_{|\alpha|} \exp\{-M(\varepsilon H^{-1} \jp{\eta}) \},
$$
where $S=\max\{\frac{K_\ell H}{\varepsilon}, \ell \}$. Given $j, \delta>0$, choose $\ell=\frac{j}{4}$ and $\varepsilon = \max\left\{\delta H, \frac{4K_jH}{j} \right\}$. Thus $S=\frac{j}{4}$ and $$\exp\{-M(\varepsilon H^{-1} \jp{\eta}) \} \leq \exp\{-M(\delta \jp{\eta}) \},$$ for all $[\eta] \in \widehat{G_2}$. Hence
$$
|\partial^\alpha \widehat{\Psi_{a}  u}(x_1,\eta)_{rs}| \leq AC_{h\delta}\left(\tfrac{j}{2}+h \right)^{|\alpha|} M_{{|\alpha|}} \exp\{-M(\delta \jp{\eta}) \},
$$
Choose now $h=\frac{j}{2}$. Therefore
$$
|\partial^\alpha \widehat{\Psi_{a}  u}(x_1,\eta)_{rs}|\leq C_{j\delta} j^{|\alpha|} M_{|\alpha|} \exp\{-M(\delta \jp{\eta}) \},
$$
which implies that $\Psi_a   u \in \Gamma_{(M_k)}(G_1\times G_2)$.
\end{proof}

\begin{prop}\label{autodualrou}
	For $a \in \Gamma_{\{M_k\}}(G_1)$, the operator $\Psi_a$ is an automorphism of $\Gamma_{\{M_k\}}'(G_1\times G_2)$.
\end{prop}
\begin{proof}

Most of the estimates that we will use here were proved in the demonstration of Theorem \ref{automrou}. Let us show that $\Psi_a   u \in \Gamma_{\{M_k\}}'(G_1 \times G_2)$ when $u \in \Gamma_{\{M_k\}}'(G_1\times G_2)$. By the characterization of ultradistributions of Roumieu type (Theorem \ref{ultraroumieu}) for all $h,\varepsilon>0$, there exists $C_{h\varepsilon}>0$ such that
$$
|\langle \widehat{u}(\: \cdot\: ,\eta)_{rs}, \varphi \rangle | \leq C_{h\varepsilon} \|\varphi\|_{h}\exp\{M(\varepsilon\jp{\eta}) \}, \quad \forall \varphi \in \Gamma_{\{M_k\}}(G_1).
$$
In this way, for $\varphi \in \Gamma_{\{M_k\}}(G_1)$, we have
\begin{align}
\jp{ \widehat{\Psi_{a}  u}(\cdot,\eta)_{rs},\varphi} = \jp{e^{i\mu_r(\eta)A(\cdot)}\widehat{u}(\cdot,\eta)_{rs},\varphi} &= \jp{\widehat{u}(\cdot,\eta)_{rs},e^{i\mu_r(\eta)A(\cdot)}\varphi }  \nonumber .
\end{align}
Hence,
$$
\jp{\widehat{u}(\cdot,\eta)_{rs},e^{i\mu_r(\eta)A(\cdot)}\varphi } \leq C_{h\varepsilon}\|e^{i\mu_r(\eta)A(\cdot)}\varphi\|_h\exp\{M(\varepsilon\jp{\eta}) \}.
$$
Notice that
$$
\left|\partial^\alpha \left(e^{i\mu_r(\eta)A(x_1)}\varphi(x_1)\right)\right| \leq \sum_{\beta\leq \alpha} \binom{\alpha}{\beta} \left|\partial^\beta e^{i\mu_r(\eta)A(x_1)}\right| \left|\partial^{\alpha-\beta }\varphi(x_1)\right|.
$$
By \eqref{propexp}, using that $|\partial^{|\alpha|} A(x_1)| \leq K \ell^{|\alpha|-1} M_{|\alpha|-1}$, we obtain
$$
|\partial^\beta e^{i\mu_r(\eta)A(x_1)}| \leq \sum_{k=1}^{|\beta|} \binom{|\beta|-1}{k-1}\frac{1}{k!} K^k\jp{\eta}^k \ell^{|\beta|-k} |\beta|! \frac{M_{|\beta|-k}}{(|\beta|-k)!}.
$$
By Proposition \ref{propM3},
$$
\jp{\eta}^{k}\exp\{M(\varepsilon\jp{\eta}) \} \leq A\varepsilon^{-k} M_{k}\exp\{M(H\varepsilon\jp{\eta}) \} ,
$$
and then by the property (M.4) we obtain
\begin{align}\label{2nd-time}
	\left|\partial^\alpha \left(e^{i\mu_r(\eta)A(x_1)}\varphi(x_1)\right)\right|\exp\{M(\varepsilon\jp{\eta}) \} &\leq A\sum_{\beta\leq \alpha} \binom{\alpha}{\beta} \sum_{k=1}^{|\beta|} \binom{|\beta|-1}{k-1}\left(\frac{K}{\ell \varepsilon}\right)^{k}  \ell^{|\beta|}M_{|\beta|} \\ & \quad \times  \left|\partial^{|\alpha|-|\beta| }\varphi(x_1)\right| \exp\{M(H\varepsilon\jp{\eta}) \} \nonumber.
\end{align}
Let $S=\max\left\{\frac{K}{\varepsilon},\ell \right\}$, then for any $j>0$ we have
\begin{align*}
\left|\partial^\alpha \left(e^{i\mu_r(\eta)A(x_1)}\varphi(x_1)\right)\right|\exp\{M(\varepsilon\jp{\eta}) \} &\leq A\sum_{\beta\leq \alpha} \binom{\alpha}{\beta}S^{|\beta|} M_{|\beta|} \left|\partial^{|\alpha|-|\beta|}\varphi(x_1)\right| \sum_{k=1}^{|\beta|} \binom{|\beta|-1}{k-1}\\ &\quad \times \exp\{M(H\varepsilon\jp{\eta}) \}  \\
&\leq A\sum_{\beta\leq \alpha} \binom{\alpha}{\beta}\left(2S\right)^{|\beta|} M_{|\beta|} \left\| \varphi \right\|_{j} j^{|\alpha|-|\beta|}M_{|\alpha|-|\beta|} \exp\{M(H\varepsilon\jp{\eta}) \}.
\end{align*}

Using the fact that $M_{|\alpha|-|\beta|}M_{|\beta|} \leq M_{|\alpha|}$ and \eqref{binom}, we obtain
\begin{align*}
\left|\partial^\alpha \left(e^{i\mu_r(\eta)A(t)}\varphi(t)\right)\right|\exp\{M(\varepsilon\jp{\eta}) \} 
	&\leq A\left(2S+j\right)^{|\alpha|} \left\| \varphi \right\|_{j} M_{|\alpha|} \exp\{M(H\varepsilon\jp{\eta}) \} \nonumber.
\end{align*}

Given $j, \delta >0$, choose $\varepsilon = \frac{\delta}{H}$ and then $h=2S+j$. Notice that
$$
\left\|e^{i\mu_r(\eta)A(\cdot)}\varphi\right\|_{h}\exp\{M(\varepsilon\jp{\eta}) \} \leq A\|\varphi\|_{j}\exp\{M(\delta\jp{\eta}) \},
$$
then we conclude that
\begin{align*}
\left|\jp{ \widehat{\Psi_{a}  u}(\cdot,\eta)_{rs},\varphi} \right| &\leq
C_{h\varepsilon}\|e^{i\mu_r(\eta)A(\cdot)}\varphi\|_h\exp\{M(\varepsilon\jp{\eta}) \} \\
&\leq C_{j\delta}	\|\varphi\|_{j}\exp\{M(\delta\jp{\eta}) \}.
\end{align*}
Therefore $\Psi_a   u \in \Gamma_{\{M_k\}}'(G_1\times G_2)$ and then $\Psi_a$ is an automorphism.
\end{proof}
\begin{prop}
	For $a \in \Gamma_{(M_k)}(G_1)$, the operator $\Psi_a$ is an automorphism of $\Gamma_{(M_k)}'(G_1\times G_2)$.
\end{prop}
\begin{proof}
	Let us show that $\Psi_a   u \in \Gamma_{(M_k)}'(G_1 \times G_2)$ when $u \in \Gamma_{(M_k)}'(G_1 \times G_2)$. By the characterization of ultradistributions of Beurling type (Theorem \ref{ultrabeurling}) there exist $h,\varepsilon, C>0$ such that
	$$
	|\langle \widehat{u}(\: \cdot\: ,\eta)_{rs}, \varphi \rangle | \leq C \|\varphi\|_{h}\exp\{M(\varepsilon\jp{\eta}) \}, \quad \forall \varphi \in \Gamma_{(M_k)}(G_1).
	$$
	In this way, for $\varphi \in \Gamma_{(M_k)}(G_1)$,
	\begin{align}
	\jp{ \widehat{\Psi_{a}  u}(\cdot,\eta)_{rs},\varphi} = \jp{e^{i\mu_r(\eta)A(\cdot)}\widehat{u}(\cdot,\eta)_{rs},\varphi} &= \jp{\widehat{u}(\cdot,\eta)_{rs},e^{i\mu_r(\eta)A(\cdot)}\varphi }  \nonumber .
	\end{align}
	We have 
	$$
	\jp{\widehat{u}(\cdot,\eta)_{rs},e^{i\mu_r(\eta)A(\cdot)}\varphi } \leq C\|e^{i\mu_r(\eta)A(\cdot)}\varphi\|_h\exp\{M(\varepsilon\jp{\eta}) \}.
	$$

	Following the proof of Proposition \ref{autodualrou}, by the fact that $a\in \Gamma_{(M_k)}(G_1)$ we obtain 
$$
\|e^{i\mu_r(\eta)A(\cdot)}\varphi\|_{2S+j}\exp\{M(\varepsilon\jp{\eta}) \} \leq A \|\varphi\|_{j}  \exp\{M(H\varepsilon\jp{\eta}) \},
$$
	where $S=\max\left\{ \frac{K_\ell}{\varepsilon},\ell\right\}.$
Now, choose $\ell=\frac{h}{4}$ and consider $\varepsilon$ sufficiently large such that $S=\ell$. For $j=\frac{h}{2}$, we obtain
\begin{align*}
	\jp{ \widehat{\Psi_{a}  u}(\cdot,\eta)_{rs},\varphi} &\leq C\|e^{i\mu_r(\eta)A(\cdot)}\varphi\|_h\exp\{M(\varepsilon\jp{\eta}) \} \\
	&\leq C \|\varphi\|_{\frac{{h}}{2}}  \exp\{M(H\varepsilon\jp{\eta}) \},
\end{align*}
which implies that $ \Psi_{a}  u \in \Gamma'_{(M_k)}(G_1 \times G_2)$.
\end{proof}

\section{Global Komatsu hypoellipticity and solvability}

Let us turn our attention to the study of global properties of the operator $L_a$ defined on the compact Lie group $G:=G_1\times G_2$ by 
$$
L_a=X_1+a(x_1)X_2,
$$
where $X_1  \in \mathfrak{g}_1$, $X_2 \in \mathfrak{g}_2$, and $a \in \Gamma_{\{M_k\}}(G_1)$ (or $a \in \Gamma_{(M_k)} (G_1)$) is a real-valued function.

The case where $a$ is a constant was studied in \cite{KMR19c} and we have the following characterization of the global properties of $L_a$:

\begin{thm} [Thms 3.2, 3.4, 3.6 and 3.8 of \cite{KMR19c}] \label{thmconstant}The operator $L_a=X_1+aX_2$, with $a\in\C$, is globally $\Gamma_{\{M_k\}}$-hypoelliptic (respectively, globally $\Gamma_{(M_k)}$-hypoelliptic) if and only if the following conditions hold:
	\begin{enumerate}[1.]
		\item The set
		$$
		\mathcal{N}=\{([\xi].[\eta])\in \widehat{G_1}\times \widehat{G_2};\  \lambda_m(\xi)+a\mu_r(\eta) = 0, \mbox{ for some } 1 \leq m \leq d_\xi \mbox{ and } 1 \leq r \leq d_\eta\}
		$$
		is finite.
		\item $\forall N>0$ (respectively, $\exists N>0$),  $\exists C_N>0$ such that
		$$
		|\lambda_m(\xi)+a\mu_r(\eta)| \geq C_N \exp\{-M(N(\jp{\xi}+\jp{\eta})) \},
		$$
		for all $[\xi] \in \widehat{G_1}$, $[\eta] \in \widehat{G_2}$, $1\leq m \leq d_\xi$, and $1 \leq r \leq d_\eta$, whenever $\lambda_m(\xi)+a\mu_r(\eta)\neq0$.
	\end{enumerate}
	Moreover, the operator $L_a$ is globally $\Gamma_{\{M_k\}}$-solvable (respectively, globally $\Gamma_{(M_k)}$-solvable) if and only if the condition 2. above is satisfied.
\end{thm}

Recall that  $L_{a_0} = X_1 + a_0 X_2$, where $a_0:=  \int_{G_1} a(x_1)\, dx_1.$ Now, if $L_{a_0}u=f \in \Gamma'_{\{M_k\}}(G)$, for some $u \in \Gamma_{\{M_k\}}'(G)$, then
$$
i(\lambda_m(\xi)+a_0\mu_r(\eta))\doublehat{u}(\xi,\eta)_{mn_{rs}} = \doublehat{\, f \,} (\xi,\eta)_{mn_{rs}},
$$ 
		for all $[\xi] \in \widehat{G_1}$, $[\eta] \in \widehat{G_2}$, $1\leq m \leq d_\xi$, and $1 \leq r \leq d_\eta$  (see \cite{KMR19c} for more details). In particular, $f$ belongs to the following set
		$$
		\mathcal{K}_{a_0}:=\{g\in\Gamma'_{\{M_k\}}(G_1 \times G_2); \  \doublehat{g}(\xi,\eta)_{mn_{rs}}=0, \mbox{ whenever } \lambda_m(\xi)+a_0\mu_r(\eta)=0\}.
		$$ 
		
In order to study the solvability of the operator $L_a$, assume that $L_a u=f \in \Gamma'_{\{M_k\}}(G_1 \times G_2)$ for some $u\in\Gamma'_{\{M_k\}}(G_1 \times G_2)$. 
We can write $u=\Psi_{-a}(\Psi_a u)$, so $L_a (\Psi_{-a}(\Psi_a u))=f$. Thus, using the fact that $
\Psi_a\circ L_a = L_{a_0}\circ\Psi_a,
$ we obtain $\Psi_{-a} L_{a_0} \Psi_au = f$, that is,
$$
L_{a_0}\Psi_{a}u = \Psi_{a}f.
$$
This implies that $\Psi_a f \in \mathcal{K}_{a_0}$ and motivates the following definition:

\begin{defi}\label{definitionsolv}
	We say that the operator $L_a$ is globally $\Gamma'_{\{M_k\}}$--solvable if $L_a(\Gamma'_{\{M_k\}}(G_1 \times G_2)) = \mathcal{J}_a$, where
$$
		\mathcal{J}_a := \{v \in \Gamma'_{\{M_k\}}(G_1 \times G_2); \ \Psi_a v \in \mathcal{K}_{a_0} \}.
		$$
\end{defi}

Similarly one defines these global properties for Komatsu classes of Beurling type. Using the results from the previous section, we obtain the following connection between the operator $L_a$ and its normal form, whose proof will be omitted because it is the same as in the smooth case (see \cite{KMR19b}).

\begin{prop}\label{normalform}
Let $a \in \Gamma_{\{M_k\}}(G_1)$ (respectively,  $a \in \Gamma_{(M_k)}(G_1)$) then:
	\begin{enumerate}[1.]
		\item the operator $L_a$ is globally $\Gamma_{\{M_k\}}$-hypoelliptic (respectively, $\Gamma_{(M_k)}$-hy\-po\-ellip\-tic) if and only if $L_{a_0}$ is globally $\Gamma_{\{M_k\}}$-hypoelliptic  (respectively, $\Gamma_{(M_k)}$-hypoelliptic);
		\item the operator $L_a$ is globally $\Gamma_{\{M_k\}}$-solvable (respectively, $\Gamma_{(M_k)}$-solvable) if and only if $L_{a_0}$ is globally $\Gamma_{\{M_k\}}$-solvable (respectively, $\Gamma_{(M_k)}$-solvable).
	\end{enumerate}
\end{prop}

From the automorphism $\Psi_{a}$ we recover for the operator $L_a$ the connection between the different notions of global hypoellipticity and global solvability, obtained in \cite{KMR19c} for constant-coefficients vector fields, summarized in the following diagram:
\begin{equation*}\label{diagram}
\begin{array}{ccccc}
GH & \Longrightarrow & G\Gamma_{\{M_k\}}H & \Longrightarrow & G\Gamma_{(M_k)}H\\
\Big\Downarrow & & \Big\Downarrow & & \Big\Downarrow\\
GS & \Longrightarrow & G\Gamma'_{\{M_k\}}S & \Longrightarrow & G\Gamma'_{(M_k)}S
\end{array}
\end{equation*}
Notice that we need to assume that $a \in \Gamma_{(M_k)}(G)$ for the implications involving Komatsu classes of Beurling type.

\subsection{Perturbations by low-order terms}
\mbox{}

We can use the results about perturbations of constant-coefficient vector fields presented in \cite{KMR19c} to study the operator $L_{aq}$ defined on $G_1 \times G_2$ by
$$
L_{aq}= X_1 + a(x_1)X_2+q(x_1,x_2),
$$
where $a \in \Gamma_{\{M_k\}}(G_1)$ is a real-valued ultradifferentiable function and $q \in \Gamma_{\{M_k\}}(G_1 \times G_2)$. The case where $a$ and $q$ are constants was presented in \cite{KMR19c}:
\begin{thm}[Thm 6.1 of \cite{KMR19c}]\label{thmperconstant}
		The operator $L_{aq}=X_1+aX_2+q$, with $a,q \in \C$, is globally $\Gamma_{\{M_k\}}$-hypoelliptic (respectively, globally $\Gamma_{(M_k)}$-hypoelliptic) if and only if the following conditions hold:
	\begin{enumerate}[1.]
		\item The set
		$$
		\mathcal{N}=\{([\xi],[\eta])\in \widehat{G_1}\times\widehat{G_2}; \lambda_m(\xi)+a\mu_r(\eta)-iq = 0, \mbox{ for some } 1 \leq m \leq d_\xi, \ 1\leq r \leq d_\eta \}
		$$
		is finite.
		\item $\forall N>0$ (respectively, $\exists N>0$), $\exists C_N>0$ such that
		$$
		|\lambda_m(\xi)+a\mu_r(\eta)-iq| \geq C_N \exp\{-M(N(\jp{\xi}+\jp{\eta})) \},
		$$
		for all $[\xi] \in \widehat{G_1}$, $[\eta] \in \widehat{G_2}$, $1\leq m \leq d_\xi$, $1\leq r \leq d_\eta$, whenever $\lambda_m(\xi)+a\mu_r(\eta)-iq\neq0$.
	\end{enumerate}
	Moreover, the operator $L_{aq}$ is globally $\Gamma_{\{M_k\}}$-solvable (respectively, globally $\Gamma_{(M_k)}$-solvable) if and only if the condition 2. above is satisfied.
\end{thm}

As discussed in \cite{KMR19c}, also previously in Remark \ref{assumption}, we will assume that there is $Q \in \Gamma_{\{M_k\}}(G_1 \times G_2)$ such that 
$$
(X_1+a(x_1)X_2)Q = q-q_0,
$$
where $q_0$ is the average of $q$ in $G_1 \times G_2$. For instance, if the operator $X_1+a(x_1)X_2$ is globally $\Gamma_{\{M_k\}}$--solvable (see Proposition \ref{normalform}) and $q-q_0$ is an admissible ultradifferentiable function, then this assumption is satisfied. We have that $e^{Q} \in \Gamma_{\{M_k\}}(G_1 \times G_2)$ and
$$
e^Q \circ L_{aq}= L_{aq_0} \circ e^Q,
$$
where $L_{aq_0} = X_1  + a(x_1) X_2 + q_0$. Now, we obtain
$$
\Psi_a \circ L_{aq_0} = L_{a_0q_0} \circ \Psi_a,
$$
where $L_{a_0q_0} = X_1  + a_0X_2 + q_0$. Therefore,
$$
\Psi_a \circ e^Q \circ L_{aq} = \Psi_a \circ L_{aq_0} \circ e^Q = L_{a_0q_0} \circ \Psi_a \circ e^{Q}.
$$

The next result is a consequence of what was done previously.

\begin{prop}\label{evopert}
	The operator $L_{aq}$ is globally $\Gamma_{\{M_k\}}$--hypoelliptic  if and only if $L_{a_0q_0}$ is globally $\Gamma_{\{M_k\}}$--hypoelliptic. Similarly, the operator $L_{aq}$ is globally $\Gamma_{\{M_k\}}$--solvable if and only if $L_{a_0q_0}$ is globally $\Gamma_{\{M_k\}}$--solvable. 
\end{prop}

We have similar results in the settings of Komatsu classes of Beurling type.

\section{Examples}

In this section we will consider the sequence $\{M_k \}_{k \in \N_0}$ given by $M_k=(k!)^s$, with $s \geq 1$. So, the Komatsu class of Roumieu type associated to this sequence is the Gevrey space $\gamma^s(G)$ and we have that the associated function satisfies
$$
M(r) \simeq r^{1/s}, 
$$ for all $r \geq0$. 

In this framework we present a class of examples in $\mathbb{T}^1 \times \St$ and in $\St \times \St$. Examples of operators defined on tori in Gevrey spaces  can be found on \cite{AKM19,BDG18}.

\subsection{$G=\mathbb{T}^1\times \St$}\label{ext1xs3}
\mbox{} 

	Consider the continued fraction  $\alpha=\left[10^{1 !}, 10^{2 !}, 10^{3 !}, \ldots\right]$ and a normalized vector field $X \in \mathfrak{s}^3$. Using rotation on $\St$, without loss of generality, we may assume that $X$ has the symbol 
	$$
	\sigma_{X}(\ell)_{mn}=im\delta_{mn}, \quad \ell \in \tfrac{1}{2}\N_0, \ -\ell\leq m,n\leq \ell, \ \ell-m, \ell-n \in \N_0,
	$$ 
	with $\delta_{mn}$ standing for the Kronecker's delta. The details about the Fourier analysis on $\St$ can be found in Chapter 11 of \cite{livropseudo}. 
	
	Consider the operator
	$$
	L_{a}=\partial_t + a(t) X,
	$$
	where $a(t)=\sin(t)+\alpha$. Notice that $a \in \gamma^s(\mathbb{T}^1)$, for all $s \geq 1$ and the function $A: t \mapsto -\cos(t)$ satisfies $\partial_tA(t)=a(t)-\alpha$. By Proposition \ref{normalform}, we can study the global properties of $L_a$ from the operator
	$$
	L_{a_0} = \partial_t + \alpha X.
	$$
	By Theorem \ref{thmconstant}, the operator $L_{a_0}$ is not globally $\gamma^s$--hypoelliptic because the set
	$$
	\mathcal{N}=\left\{(k,\ell) \in \Z \times \tfrac{1}{2}\N_0; \ k+\alpha m = 0, \mbox{ for some } -\ell \leq m \leq \ell, \ \ell-m \in \Z \right\}
	$$
	has infinitely many elements. However, since $\alpha$ is not an exponential Liouville number of order $s$, for any $s\geq 1$, for all $N>0$ there exists $C_N>0$ such that
	$$
	|k+\alpha m| \geq C_N \exp\{-N (|k|+ \ell +1)^{1/s} \},
	$$
	for all $k \in \Z$, $\ell \in \N_0$, $-\ell  \leq m \leq \ell$, $\ell -m \in \N_0$, whenever $k+\alpha m \neq 0$. Therefore, the operator $L_{a_0}$ is globally $\gamma^s$--solvable, for any $s \geq 1$. In addition, since $\alpha$ is a Liouville number, the operator $L_{a_0}$ is not globally solvable in the $C^\infty$--sense.
	
	We conclude then that the operator $L_a$ is neither globally $\gamma^s$-hypoelliptic, nor globally solvable, but it is globally $\gamma^s$--solvable, for any $s\geq 1$.
	
    Consider now
	$$
	L_{aq}=\partial_t+a(t)X+q(t,x)
	$$
	where $X\in \mathfrak{s}^2$, $a(t)=\sin(t)+\alpha$, and  $q(t,x)=\cos(t)+(\sin(t)+\alpha)h(x)+\frac{1}{2}i $, where $h$ is expressed in Euler's angle by
	\begin{equation*}
	h(x(\phi,\theta,\psi)) = -\cos\left(\tfrac{\theta}{2}\right)\sin\left(\tfrac{\phi+\psi}{2}\right),
	\end{equation*}
	where $0 \leq \phi < 2\pi$, $0 \leq \theta \leq \pi$, $-2\pi \leq \psi < 2\pi$.	Notice that $q$ is an analytic function, which implies that $q \in \gamma^s(\mathbb{T}^1\times \St)$ for all $s \geq 1$. 
	
	The vector field $X$ is the operator $\partial_{\psi}$ in Euler's angle  and we have that $X {\textrm{tr}}(x) = h(x)$, where the trace function ${\mbox{tr}}$ is expressed in Euler's angle  by
	$$
	{\mbox{tr}}(x(\phi,\theta,\psi)) = 2\cos\left(\tfrac{\theta}{2}\right)\cos\left(\tfrac{\phi+\psi}{2}\right).
	$$

	The function $Q(t,x)=\sin(t)+{\mbox{tr}}(x)$ satisfies
	$$
	(\partial_t+a(t) X) Q(t,x)=q(t,x)-\tfrac{1}{2}i.
	$$
	By Proposition \ref{evopert}, the operator
	$$
	L_{aq} = \partial_t + (\sin(t)+\alpha)X + \left\{ \cos(t)+(\sin(t)+\alpha)h(x)+\tfrac{1}{2}i \right\}
	$$
	is globally $\gamma^s$--hypoelliptic if and only if
	$$
	L_{a_0q_0} = \partial_t + \alpha X + \tfrac{1}{2}i
	$$
	is globally $\gamma^s$--hypoelliptic. By Example 6.7 of \cite{KMR19c}, we conclude that  $L_{aq}$ is globally $\gamma^s$--hypoelliptic for any $s \geq 1$, which implies that it is also globally $\gamma^s$--solvable, for any $s \geq 1$. In addition, the operator $L_{aq}$  is neither globally hypoelliptic nor globally solvable in $C^\infty$--sense, because $L_{a_0q_0}$ has these properties.
	
	Similarly, the operator 	
	$$
	L_{aq} = \partial_t + (\sin(t)+\alpha)X + \left\{ \cos(t)+(\sin(t)+\alpha)h(x)+\alpha i \right\}
	$$
	is not globally $\gamma^s$--hypoelliptic but is globally $\gamma^s$--solvable because
	$$
	L_{a_0q_0} = \partial_t + \alpha X + \alpha i
	$$
	has these properties. Again, the operator $L_{aq}$ is neither globally hypoelliptic nor globally solvable in the $C^\infty$--sense.

\subsection{$G=\St\times \St$}\label{exs3xs3}
\mbox{}

Consider the operator
	$$
	L_{h} = X_1 + h(x_1)X_2,
	$$
	where $X_1, X_2 \in \mathfrak{s}^3$, $h$ is expressed in Euler's angle by
	\begin{equation*}
	h(x_1(\phi_1,\theta_1,\psi_1)) = -\cos\left(\tfrac{\theta_1}{2}\right)\sin\left(\tfrac{\phi_1+\psi_1}{2}\right)+ \alpha,
	\end{equation*}
	where $0 \leq \phi_1 < 2\pi$, $0 \leq \theta_1 \leq \pi$, $-2\pi \leq \psi_1 < 2\pi$, and $\alpha$ is the continued fraction  $\left[10^{1 !}, 10^{2 !}, 10^{3 !}, \ldots\right]$. Moreover, we will assume that the vector field $X_1$ acts only in the first variable, while $X_2$ acts only in the second variable. In this way, we may assume that
	$$
	\sigma_{X_1}(\ell)_{mn}=im\delta_{mn}, \quad \ell \in \tfrac{1}{2}\N_0, \ -\ell\leq m,n\leq \ell, \ \ell-m, \ell-n \in \N_0,
	$$ 
	and
	$$
	\sigma_{X_2}(\kappa)_{rs}=ir\delta_{rs}, \quad \kappa \in \tfrac{1}{2}\N_0, \ -\kappa\leq r,s\leq \kappa, \ \kappa-r, \kappa-s \in \N_0.
	$$ 
So, the $X_j$ is the operator $\partial_{\psi_j}$ in Euler's angles, for $j=1,2.$ Since $X_1 \mbox{tr}(x_1) = h(x_1)-\alpha$, with $\mbox{tr}$ as in Example \ref{ext1xs3},
it is enough to understand the global properties of the operator
	$$
	L_{h_0}=X_1 + \alpha X_2
	$$
	for the study of the global properties of $L_{h}$. By Theorem \ref{thmconstant}, the operator $L_{h_0}$ is not globally $\gamma^s$--hypoelliptic because the set
	$$
	\mathcal{N}=\left\{(\ell,\kappa) \in \tfrac{1}{2}\N_0\times \tfrac{1}{2}\N_0 ; \ m+\alpha r=0, \mbox{ for some } -\ell \leq m \leq  \ell, \ -\kappa \leq r \leq \kappa \right\} 
	$$ 
	has infinitely many elements. However, since $\alpha$ is not an exponential Liouville number or order $s$, for any $s \geq 1$, for all $N>0$ there exists $C_N>0$ such that
	$$
	|m + \alpha r| \geq C_N\exp\{-N(\ell+\kappa+1)^{1/2} \},
	$$
	for all $\kappa, \ell \in \tfrac{1}{2}\N_0$, $-\ell \leq m \leq \ell$, $-\kappa \leq r \leq \kappa$, whenever $m+\alpha r \neq 0$. Thus, the operator $L_{h_0}$ is globally $\gamma^s$--solvable, for any $s \geq 1$. Furthermore, $L_{h_0}$ is not globally solvable in the $C^\infty$--sense because $\alpha$ is a Liouville number.
	Therefore, the operator $L_h$ is neither globally $\gamma^s$--hypoelliptic, nor globally $C^\infty$--solvable, but it is globally $\gamma^s$--solvable, for any $s \geq 1$.

Consider now  the operator
	$$
	L_{hq} = X_1 + h(x_1)X_2 + q(x_1,x_2),
	$$
	where $q$ is given by
	$$
	q(x_1,x_2)=p_1(x_1)+h(x_1)p_2(x_2)+\tfrac{1}{2}i,
	$$
	where $p_1$ and $p_2$ are the projections of ${\mbox{SU}(2)} \simeq \St$ given in Euler's angle by
	$$
	p_1(x(\phi,\theta,\psi))= \cos\left(\tfrac{\theta}{2}\right)e^{i(\phi+\psi)/2}
	\quad \mbox{and} \quad
	p_2(x(\phi,\theta,\psi))= i\sin\left(\tfrac{\theta}{2}\right)e^{i(\phi-\psi)/2},
	$$
	where $0 \leq \phi < 2\pi$, $0 \leq \theta \leq \pi$, $-2\pi \leq \psi < 2\pi$.
 It is easy to see that the function $Q(x_1,x_2) = 2i(p_2(x_2)-p_1(x_1))$ satisfies
	$$
	(X_1+h(x_1)X_2)Q(x_1,x_2)=q(x_1,x_2)-\tfrac{1}{2}i.
	$$
	Since $Q$ is analytic, we have that $Q \in \gamma^s(\St\times \St)$, for any $s \geq 1$. By Proposition \ref{evopert}, we can extract the global properties of $L_{hq}$ from the operator
	$$
	L_{h_0q_0}=X_1+\alpha X_2 + \tfrac{1}{2}i.
	$$
As in Example \ref{ext1xs3}, we conclude by Theorem \ref{thmperconstant} that the operator $L_{h_0q_0}$ is globally $\gamma^s$--hypoelliptic for any $s\geq 1$, but is not globally solvable in the $C^\infty$--sense.  By Proposition \ref{evopert}, the operator $L_{hq}$ has the same properties of $L_{h_0q_0}$.

\section*{Acknowledgments}

This study was financed in part by the Coordena\c c\~ao de Aperfei\c coamento de Pessoal de N\'ivel Superior - Brasil (CAPES) - Finance Code 001. The last
author was also supported by the FWO Odysseus grant, by the Leverhulme Grant RPG-2017-151, and by EPSRC Grant EP/R003025/1.
\bibliographystyle{abbrv}
\bibliography{biblio}

\end{document}